\documentclass{amsart}
\usepackage{amsfonts,latexsym,amssymb,amsmath,amsthm,color}
\usepackage{enumerate}
\usepackage{pb-diagram}
\usepackage{captdef}
\usepackage[dvips]{graphicx} 
\usepackage{enumerate,graphicx,graphpap}
\usepackage[matrix]{xy}
\input xy
\xyoption{all} \xyoption{poly}

\newtheorem{teorema}{Theorem}

\newtheorem{lema}[teorema]{Lemma}

\theoremstyle{definition}
\newtheorem{defin}[teorema]{Definition}

\theoremstyle{remark}

\newcommand{\be}{\begin{enumerate}}
\newcommand{\ee}{\end{enumerate}}

\newcommand{\bi}{\begin{itemize}}
\newcommand{\ei}{\end{itemize}}

\newcommand{\NN}{{\mathbb N}}

\newcommand{\CC}{{\mathbb C}}

\newcommand{\FF}{{\mathcal F}}

\newcommand{\om}{\omega}

\newcommand{\bx}{\mathbf{x}}

\newcommand{\bo}{{\mathbf 0}}

\DeclareMathOperator{\Sing}{Sing}

\begin{document}

\title{Analytic classification of a class  of cuspidal foliations in
$(\CC^3,0)$}

\author{Percy Fern\'{a}ndez-S\'{a}nchez}
\address[Percy Fern\'{a}ndez]{Dpto. Ciencias - Secci\'{o}n Matem\'{a}ticas, Pontificia Universidad Cat\'{o}lica del Per\'{u}, Av. Universitaria 1801,
San Miguel, Lima 32, Peru}
\email{pefernan@pucp.edu.pe}

\author{Jorge Mozo-Fern\'{a}ndez}
\address[Jorge Mozo Fern\'{a}ndez]{Dpto. \'{A}lgebra, An\'{a}lisis Matem\'{a}tico, Geometr\'{\i}a y Topolog\'{\i}a \\
Facultad de Ciencias, Universidad de Valladolid \\
Campus Miguel Delibes\\
Paseo de Bel\'{e}n, 7\\
47011 Valladolid - Spain}
\email{jmozo@maf.uva.es}

\author{Hern\'{a}n Neciosup}
\address[Hern\'{a}n Neciosup]{Dpto. Ciencias - Secci\'{o}n Matem\'{a}ticas, Pontificia Universidad Cat\'{o}lica del Per\'{u}, Av. Universitaria 1801,
San Miguel, Lima 32, Peru}
\email{hneciosup@pucp.pe}

\thanks{First and third authors partially supported by the Pontificia Universidad Cat\'{o}lica del Per\'{u} project VRI-DGI 2014-0025. \\ Second author partially supported by the Ministerio de Econom\'{\i}a y Competitividad from Spain, under the Project ``\'{A}lgebra y Geometr\'{\i}a en Din\'{a}mica Real y Compleja III" (Ref.: MTM2013-46337-C2-1-P)}


\date{\today}
\begin{abstract}
In this article we study the analytic classification of certain types of quasi-homogeneous cuspidal holomorphic foliations in $(\CC^3,{\bf 0})$ via the essential holonomy defined over one of the components of the exceptional divisor that appears in the reduction of the singularities of the foliation.
\end{abstract}

\maketitle
\section{Introduction} \label{introduccion}
The objective of this paper is to give a contribution to the analytical classification of germs of codimension one holomorphic foliations defined on an ambient space of dimension three. More generally, assume that the dimension of the ambient space is arbitrary, call it $n$.
Let $\FF_1, \FF_2$ be two such germs, generated  by integrable 1-forms $\omega_1, \omega_2$, respectively, holomorphic in a neighbourhood of $\bo\in \CC^n$. We will say that $\FF_1, \FF_2$ are analytically equivalent if there exists a germ of diffeomorphism $\phi:(\CC^n, {\bf 0})\to (\CC^n, {\bf0})$ such that $\phi^*\omega_1\wedge\omega_2=0$.

Let us mention here some of the previous results concerning this subject. When $n=2$, J. Martinet and J.-P. Ramis \cite{MR1,MR2} consider simple (reduced) singularities, both in the saddle-node case and in the resonant case. Concerning non-simple singularities, consider the  nilpotent case (i.e. foliations defined by a 1-form $\omega$ such that its dual vector field has nilpotent, non-zero, linear part), in which these foliations admit a formal normal form $d(y^2+x^n)+x^p.u(x)dy$, where $n\geq3, p\geq2$ \cite{Takens}, and $u(0)\neq0$, and in fact an analytic pre-normal form $\omega=d(y^2+x^n)+A(x,y)(nydx-2xdy)$ (see \cite{Cerveau-Moussu}). These foliations have been the object of study of R. Moussu \cite{R.Moussu2}, D. Cerveau \cite{Cerveau-Moussu}, R. Meziani \cite{Meziani}, M. Berthier, P. Sad \cite{BMS} and E. Str\'{o}\.{z}yna \cite{Strozyna}. It is worth to mention here other previous works that contribute to the analytic classification of holomorphic foliations, as the ones by J.-F. Mattei \cite{Mattei-modules,Mattei}  and  Y. Genzmer \cite{Genzmer}. Other authors focus in the topological classification, that we will not treat in this paper. Among them, let us mention, without trying to be exhaustive, works of D. Mar\'{\i}n and J.-F. Mattei \cite{MarinMattei}, and of R. Rosas \cite{Rosas}.

In dimension $n=3$, simple singularities have been studied by D. Cerveau and J. Mozo in \cite{Cerveau-Mozo}. Concerning non-simple singularities, P. Fern\'{a}ndez and J. Mozo studied in \cite{FM} the case of quasi-ordinary, cuspidal singularities. In these works, the main tool used is, either the holonomy of one of the separatrices of the foliation or, more frequently, the projective holonomy of a certain component of the exceptional divisor that appears after the reduction of the singularities. In order to apply this technique, specially in dimension higher than two, it is of great importance to study the topology of the divisor, as the holonomy group of each component is a representation of its fundamental group.

In this paper, we will apply this technique to a  kind of cuspidal foliations more general than the quasi-ordinary ones studied in \cite{FM}, more precisely, admissible quasi-homogeneous foliation (see Definition \ref{def_quasi_homogeneous} in Section \ref{general}) of generalized surface type. Recall, from \cite{FM2} and \cite{FMNeciosup1}, that a germ $\FF$ of non dicritical foliation in $(\CC^3,{\bf 0})$, generated by an integrable 1-form $\omega$, is a generalized surface if every generically transverse plane section $\varphi:(\CC^2,{\bf 0})\to (\CC^3,{\bf 0})$ (i.e., $\varphi^{*}\omega\not\equiv0 $) defines a generalized curve in the sense of \cite{CLS}. In particular, for generalized surfaces, once its set of separatrices has been reduced, the foliation has only simple singularities, as shown in \cite{FM2} and generalized in \cite{FMNeciosup1} for higher dimensions.

In Section \ref{general}, a cuspidal foliation in $(\CC^3,{\bf 0})$ will be defined as a generalized surface $\FF$ having a separatrix written as $z^2+\varphi(x,y)=0$, in appropriate coordinates. We will consider in this paper the situation in which the separatrix is quasi-homogeneous. The first point would be to find an analytic normal form for these type of foliations, in the style of the ones considered by D. Cerveau and R. Moussu in dimension two. This is the main problem studied in \cite{FMNeciosup2}. In that paper, it is shown that a cuspidal, quasi-homogeneous foliation $\FF$ with a separatrix $z^2+\varphi(x,y)=0$, can be generated by an integrable 1-form
$$\omega=d(z^2+\varphi)+G(\Psi,z)\cdot z\cdot\Psi\cdot\left(2{dz\over z}-{d\varphi\over\varphi}\right),$$
where $\Psi^r=\varphi, \Psi$ is not a power, and $G$ is a germ of holomorphic function in two variables. This is the starting point of the present paper.

Also in Section \ref{general}, the main definitions of the paper are presented. We will work with a concrete reduction of singularities of cuspidal quasi-homogeneous foliations, obtained via the Weierstrass-Jung method. For the sake of completeness, this reduction will be sketched in Section \ref{section3}. In Section \ref{section_toplogy_of_the_divisor}, the topology of the components of the exceptional divisor is studied. As we said before, this is important in order to identify the component of the exceptional divisor  where the holonomy is computed. In particular there is one component such that, once removed the singular locus, it is homeomorphic to $(\CC^*\times\CC )\smallsetminus\mathcal{C}$, where $\mathcal{\mathcal{C}}$ is a certain algebraic plane curve, whose topology will be interesting for us. The fundamental group will be computed in this section.

In order to lift the holonomy of a component of the divisor, a fibration transverse to the foliation outside the separatrices, is constructed in Section \ref{section_holonomy_of_the_essential_component}. Finally, Section \ref{section_cladification}, is devoted to study the main problem of the paper: to classify analytically quasi-homogeneous generalized surfaces following this method. Unfortunately, the answer is not complete, and we must impose some conditions (properties $\wp_1, \wp_2$ of Definition \ref{defini_propiedades_p_1andp_2}), that guarantee the linearization of the holonomy over some ``end component'' of the divisor, that we will call \textit{special components}. Under these conditions, Theorem \ref{thm_clasificacion_analitica} of this section is the main result of the paper, and gives the analytic classification.

As a future work, we would like to study a wider class of foliations. The existing works of Mattei and Genzmer \cite{Mattei-modules, Mattei,Genzmer} in dimension two lead us to think that different techniques should be developed, as for instance, techniques concerning deformations of foliations. We will not follow this method here.

\section{Generalities} \label{general}
In dimension $n\geq2$, a germ of codimension one holomorphic foliation $\FF$ is called \textit{cuspidal} if, besides being generalized hypersurface (see Section \ref{introduccion} and \cite{FM2}), its separatrix is defined in appropriate coordinates $(\bx, z)$, where $\textbf{x}:=(x_1,x_2,\cdots,x_{n-1})$, by a reduced equation $z^2+\varphi(\textbf{x})=0$,  with $\nu_0 (\varphi (\textbf{x}))\geq 2$. In \cite[Thm. 1.4]{FMNeciosup2}  we constructed a pre-normal form for germs of holomorphic foliation of cuspidal type. In particular, if $z^2+\varphi(\textbf{x})=0$ is a quasi-homogeneous hypersurface, we know that there exist analytic coordinates such that a generator of $\FF$ is the 1-form
$$
\om = d(z^2+\varphi) + G(\Psi, z) (z\cdot \Psi)\cdot \left(
2\frac{dz}{z}-\frac{d\varphi}{\varphi}\right) ,
$$
where $\varphi= \Psi^r$ for some $r\in\NN,\;\Psi$ is not a power, and $G$ is a germ of holomorphic function in two variables (\cite[Cor. 4.2]{FMNeciosup2}). \\

In dimension three, there exist analytic coordinates $(x,y,z)$ such that
$$\Psi=x^{n_1}y^{n_2}\displaystyle\prod_{i=1}^{l}(y^p-a_ix^q)^{d_i},$$
where $n_i,p,q\in\mathbb{N}$, and $a_i\in \CC^*$ are pairwise distinct. After \cite[Thm. 2.4]{FMNeciosup2} and \cite[Thm. 11]{FM2}, the reduction of the singularities of this type of foliations follows the same scheme that the desingularization of their separatrix $S$,
$$S:=z^2+x^{n_1}y^{n_2}\displaystyle\prod_{i=1}^{l}(y^p-a_ix^q)^{d_i}=0.$$

A key to prove the main result of this paper (Theorem \ref{thm_clasificacion_analitica}) is the existence of the first integral in the intersections of the components of the exceptional divisor.
Let us mention in dimension two the work \cite{D.Marin-Thesis}, where it is said that the classification technique from R. Moussu \cite{R.Moussu2} and F. Loray \cite{Loray-Thesis} can be used in the quasi-homogeneous framework, \textit{provided that the coordinate axis are not part of this set of separatrices}. i.e., after reduction of singularities, all the separatrices lie on the same component of the exceptional divisor. Because of this, we will assume an analogous hypothesis in dimension three, i.e., $n_1=n_2=0$, 
and Theorem \ref{thm_clasificacion_analitica} will be proved based on this property. This means that coordinate hyperplanes are not part of its set of separatrices. Let us mention that on the other hand, in \cite{Genzmer} this hypothesis is not considered, but other techniques of a rather different nature have been used, which will not be addressed here. According to this, let us state the following definition:

\begin{defin}\label{def_quasi_homogeneous}
A cuspidal holomorphic foliation, $\FF$, in $(\CC^3,{\bf 0})$ is called {\bf quasi-homogeneous of admissible type}, if its separatrix, in appropriate coordinates, is defined by the equation
$$S=z^2+\displaystyle\prod_{i=1}^{l}(y^p-a_ix^q)^{d_i}=0,$$
where $p,q\geq2,\;a_i\in\CC^*$,  and $a_i\neq a_j,$ if $i\neq j$.
\end{defin}

So, a quasi-homogeneous foliations of admissible type is generated by a 1-form
\begin{equation}\label{forma_quasi_homogeneous}
\Omega=d(z^2+\varphi)+G(\Psi,z)\cdot z\cdot\Psi\left({d\varphi\over\varphi}-2{dz\over z}\right),
\end{equation}
where $\varphi=\Psi^r=\displaystyle\prod_{i=1}^{l}(y^p-a_ix^q)^{d_i},\;p,q\geq2,\; r=\gcd (d_1,\cdots,d_l)$.

In general, let us denote by $\FF_{\omega}$ a holomorphic foliation of codimension one generated by an integrable 1-form $\omega\in\Omega^1(\CC^n,{\bf 0})$.
Not every such 1-form defines a quasi-homogeneous foliation of admissible type, as we have imposed to the definition the condition of being a generalized surface. In general, such a foliation may have more separatrices, or be dicritical. In \cite[Thm. 5.2]{FMNeciosup2} a sufficient condition is stated for being a generalized surface. More precisely, if we write $G(\Psi,z)=\displaystyle\sum_{\alpha,\beta}G_{\alpha,\beta}\Psi^{\alpha}z^{\beta}$, and
$$\nu_{2,r}(G):=\min\left\{{2\alpha+r\beta\over \gcd(2,r)}; G_{\alpha,\beta}\neq0\right\},$$
then, if $\nu_{2,r}(G)\geq{r-2\over \gcd(2,r)}$, the foliation is a generalized surface.

On the other hand, let us observe that if $d_i=1$ for all $i$, $\varphi$ is reduced and the singular locus of $\FF_{\Omega}$ is the origin of coordinates. By Frobenius  singular Theorem \cite{Malgrange}, $\FF_{\Omega}$ admits holomorphic first integral and the study of these foliations is then equivalent to that of the surfaces. We will suppose, at any time, that $d_i>1$ for some $i$, and we will denote $\Sigma_{p,q}^{(d_1,\cdots,d_l)}$ the set of integrable 1-forms  analytically equivalent to a 1-form  as in (\ref{forma_quasi_homogeneous}), with $d_i>1$ for some $i$.

\section{Desingularization of quasi homogeneous foliations}\label{section3}\label{Desingularization}

Let us consider a quasi-homogeneous foliation of admissible type, $\FF_{\Omega}$, generated by an integrable 1-form $\Omega\in\displaystyle \Sigma_{p,q}^{d_1,\cdots,d_l}$, according with the notations of Section \ref{general}. The reduction of the singularities of $\FF_{\Omega}$ is achieved after the reduction of its separatrices. We  use in this paper a precise reduction, namely the one obtained by Weierstrass-Jung method, that we shall sketch here for the sake of completeness. This precise reduction will be useful in the sequel in order to extend the conjugation of the holonomy to a whole neighbourhood of the exceptional divisor.

This reduction will be described in three steps:

{\bf Step I}: The separatrix $S$ being defined by $z^2+\varphi(x,y)=z^2+\displaystyle\prod_{i=1}^{l}(y^p-a_ix^q)^{d_i}=0$, we shall proceed to desingularize first the curve $\varphi(x,y)=0$. This is done algorithmically taking into account its characteristic exponents, more precisely, the continuous fraction expansion of $\displaystyle{p\over q}=[c_0;c_1,\cdots,c_N]$. It is necessary to do $k=\displaystyle\sum_{\nu
=0}^{N}c_{\nu}$ quadratic transformations (point blow-ups), whose composition will be denoted $\pi_I:$

$$\pi_I:(M_I,E_I)\to(\CC^3,{\bf 0}).$$

Denote $\widetilde{\Omega}$ the strict transform of $\Omega$. Locally, in appropriate coordinates, it is given as
$$\widetilde{\Omega}=\displaystyle(z^2+x^ay^bU_{\alpha})\omega_{\alpha}+xy\eta_{\alpha},$$ where $a,b$ are natural numbers, depending on $p,q$ and the chosen chart, $\omega_{\alpha}=mydx+nxdy$ is a linear form, with $m$, $n\in \NN$, $U_{\alpha}$ is a holomorphic function, that in the ``most interesting'' chart obtained after the last blow-up is
$\displaystyle U_{\alpha}=h^r=\left(\prod_{i=1}^l(y^{\delta}-a_i)^{d_i'}\right)^r$, where  $r=\gcd(d_1,\cdots,d_l)$, $d_i'=d_i/r$, $\delta=\gcd (p,q)$ and $\displaystyle\eta_{\alpha}=d\left(z^2+x^ay^bU_{\alpha}\right)+\Delta_{\alpha}\cdot \left(a{dx\over x}+b{dy\over y}+{dU_{\alpha}\over U_{\alpha}}-2{dz\over z}\right)$. Here $\Delta_{\alpha}$  is a certain germ of holomorphic function that we will not precise here.

{\bf Step II:} Consider the foliation defined by $\widetilde{\Omega}$. We will blow-up certain curves biholomorphic to $\mathbb{P}_{\CC}^1$. This process will depend of the nature of integers $a$, $b$, and follows the scheme of the quasi-ordinary case studied in \cite{FM}. More precisely:

{\bf II.a) If $d$ is even,} $a, b$ are also even. Blow-up $(z=y=0)\:\; {a\over 2}$ times and $(z=x=0)\:\; {b\over 2}$ times. The final result is schematized in Figure \ref{figura2}.
   \begin{center}
     \includegraphics[width=11.5cm]{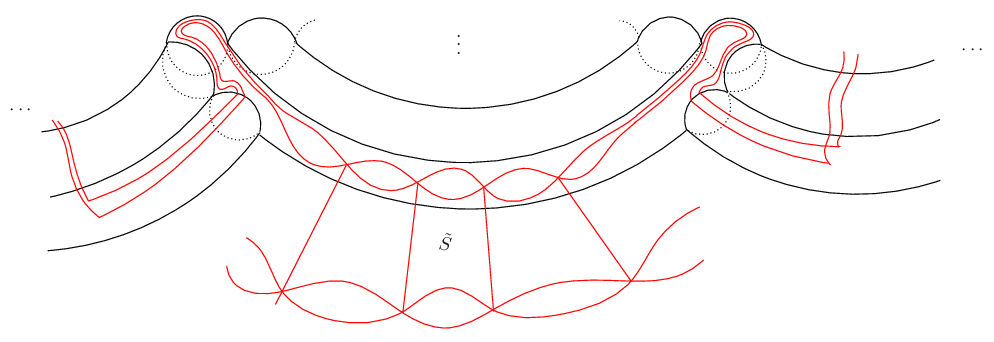}\\
      \figcaption{$\widetilde{S}$, strict transform by $\pi_{_{II}}$ of $S$. }\label{figura2}
   \end{center}
{\bf II.b) If $d$ is odd, $p$ even and $q$ odd}, a certain number of suitable monoidal transformations, depending on the chart, is necessary to perform. The final result is schematized in Figure \ref{figuraCasoii_1}.
\begin{center}
    \includegraphics[width=10.5cm]{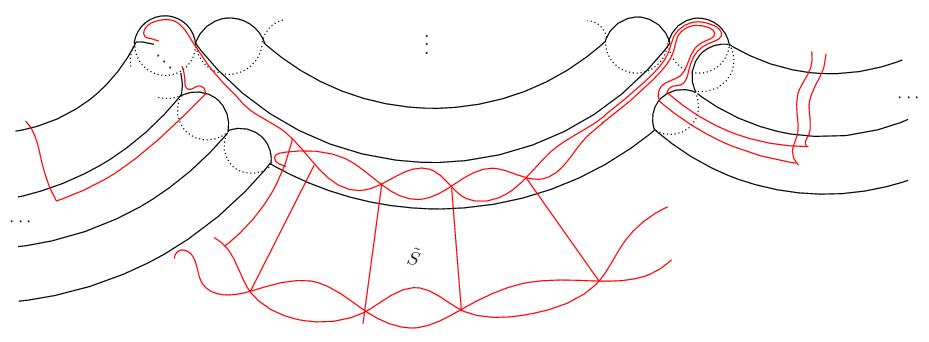}\\
    \figcaption{$\widetilde{S}$, strict transform by $\pi_{II}$ of $S$ (Case II.b)}\label{figuraCasoii_1}
    \end{center}
{\bf II.c) If $d, a$ and $b$ are odd}, after the sequence of monoidal transformations, a final quadratic transformation will be necessary to do, as in the quasi-ordinary case of \cite{FM}. Figure \ref{figure_trans_monoidal_caso_impar} represents the final result in an appropriate chart.
    \begin{center}
    \includegraphics[width=11cm,height=3.5cm]{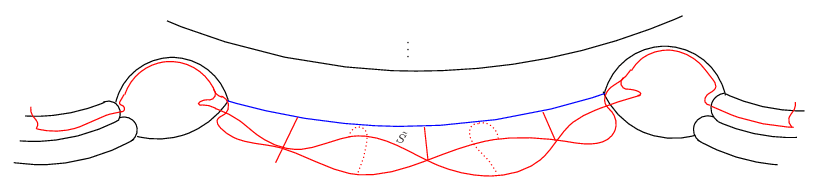}\\
    \figcaption{$\widetilde{S}$, strict transform by $\pi_{II}$ of $S$ (Case II.c)}\label{figure_trans_monoidal_caso_impar}
    \end{center}
    We shall denote $\pi_{II}$ the composition of the transformations done in this step.

{\bf Step III}. At the end of Step II, there exist analytic coordinates in which the local equation of the strict transform of the foliation is
$$\Omega_{PQ}=(t^2+h^r)\omega_{PQ}+xy\eta_{PQ},$$
where
$$\begin{array}{rcl}
 \omega_{_{PQ}}    & = & ({pq\over\delta}d)ydx+(nqd)xdy\\
 \eta_{_{PQ}}      & = & d(t^2+h^r)+x^{{pq\over \delta}d'(1-{r\over2})}y^{nqd'(1-{r\over 2})}th.G_{1PQ}\Bigg(r{dh\over h}-2{dt\over t}\Bigg)\\
    G_{1PQ}        & = & G\Big(x^{{pq\over\delta}d'}y^{nqd'}h,x^{{pq\over
                         \delta}{d\over2}}y^{nq{d\over2}}t\Big).\\
    P              & = & {pq\over\delta}d-2\left({p+q\over\delta}-1\right).\\
    Q              & = & nqd-(m+n-1).
\end{array}$$
It is necessary to blow-up the lines $z=0, y=a_i^{1\over \delta}$, according to the nature of each $d_i$. The components of the divisor are, topologically, $\mathcal{A}_j\times \mathbb{D}$, where $\mathcal{A}_j\approx\mathbb{P}_{\CC}^1$ and $\mathbb{D}$ a disc around the origin.

Denote $\widetilde{D}$ the last component generate in Step II, i.e., the main component where the separatrix cuts the exceptional divisor. This component will be called  {\it essential component}.
Note that the modifications done in this step do not alter the topology of this component. This will be precised in Section \ref{section_toplogy_of_the_divisor}.

Analogously as in the previous steps, $\pi_{III}$ will denote the composition of the transformations done here.

\section{Topology of the divisor}\label{section_toplogy_of_the_divisor}

Let $\widetilde{\FF}$ be the strict transform of a quasi-homogeneous holomorphic foliation of admissible type via the morphism $\pi:(M,E)\to(\CC^3,{\bf0})$ ($\pi=\pi_{_{I}}\circ\pi_{_{II}}\circ\pi_{_{III}}$), as described in Section \ref{section3}.

With the conditions imposed, the singular locus $\mathcal{S}:=\Sing(\widetilde{\FF})$, is an analytic space of dimension $1$  with normal crossings. $\mathcal{S}$ is given by the intersection of the components of the divisor $E$, along with the intersection of the strict transform of $S$ with the divisor. These components are denoted as $D_{\alpha},\; D_{\alpha j},\; \mathcal{A}_j\times \mathbb{D}$ (see Section \ref{section3} for notations).
It will be useful, in order to study the projective holonomy of the foliation, to know their homotopy type, once we have removed the singular locus. Let us describe it briefly in this section.

In the first step of the reduction we have produced several components $D_{\alpha}$. After removing the singular locus they are homotopically equivalent to:
$$D_{\alpha}\smallsetminus\mathcal{S}\approx
\left\{
  \begin{array}{ll}
    \CC\times\CC, & \hbox{if}\; \alpha=1;  \\\\
    \CC\times\CC^*, & \hbox{if}\;1<\alpha\leq c_0+1; \\\\
    \CC^*\times\CC^*, & \hbox{ in other cases.}
  \end{array}
\right.
$$

The components $D_{\alpha j}$ appearing in the second step, removing the singular locus, are homotopically equivalent to $\CC\times\CC^*$, to $\CC^*\times\CC^*,$ to $\CC\times(\CC\smallsetminus\{\text{2 points}\})$, or, in the most interesting case, to $(\CC^*\times\CC)\smallsetminus\mathcal{C}$, where, in coordinates $(y,t), \mathcal{C}$ is the affine curve defined by $t^2-y^{a}.v=0$, where $v$ is a unit, or
$$t^2-\displaystyle\left(\prod_{i=1}^{l}(y^{\delta}-a_i)^{d'_i}\right)^{r}=0,$$
with previous notations.

Finally, Step III produces new components homotopically equivalent to $\CC\times\CC$, $\CC^*\times\CC$ or $(\CC\smallsetminus\{\text{2 points}\} )\times\CC$, that are not relevant for our study. However, we highlight the component $D_{\alpha j}$ (generated in the second step) which is modified now to $\widetilde{D}$ (essential component defined in Section \ref{Desingularization}), but whose topology does not change in this step and consequently $\widetilde{D}\smallsetminus\Sing(\widetilde{\FF}_{\Omega})\thickapprox (\CC^*\times\CC )\smallsetminus\mathcal{C}$.

The fundamental group of $(\CC^*\times\CC)\smallsetminus\mathcal{C}$ can be computed using the Zariski-Van Kampen method, as described, for instance, in \cite{ACT} (see also \cite{VK} as a classical reference). Even if it is standard enough, for the sake of completeness we shall briefly review here this method. Consider the projection $\rho:\CC^2\to\CC,\;\rho(x,t)=x$.
Let us denote by $\Delta:=\{a_1,a_2,\cdots,a_l\}$ and $\mathcal{L}:=\displaystyle\bigcup_j\rho^{-1}(a_j)$. The restriction \begin{equation}\label{fibracion_localmente_trivial}
\rho:\CC^2\smallsetminus\mathcal{C}\cup\mathcal{L}\to\CC\smallsetminus\Delta,
\end{equation}
is a locally trivial fibration, with fibres isomorphic to $\CC\smallsetminus\{\text{2 points}\}$. So, for every point in $\CC\smallsetminus\Delta$, we obtain the exact sequence
$$
\xymatrix{1\ar[r] & \pi_1\left(\CC\smallsetminus\{\text{2 points}\}\right)\ar[r]^-{i_{*}} &  \pi_1(\CC^2\smallsetminus \mathcal{C}\cup\mathcal{L})  \ar[r]^-{\rho_{*}}& \ar@/_2pc/[l]^-{s_{*}}\pi_1(\CC\smallsetminus\Delta)\ar[r]& 1},
$$
where $i^*$ and $s^*$ are map induced in homotopy from inclusion and section of $\rho$ respectively.
Let $\mathbb{D}\subset \CC$ be a sufficiently big closed disk, such that $\Delta$ is contained in its interior. Choose a point $\star$ on $\partial \mathbb{D}$. Take a small disk $\mathbb{D}_j$ centered at ${a_j}\in\Delta$ containing no other elements of $\Delta$ and choose a point $R\in\partial\mathbb{D}_j$. Consider a path $\alpha\in\CC\smallsetminus\Delta$ joining $R$ and $\star$, and denote by $\eta_{R,\mathbb{D}_j}$ the closed path based at $R$ that runs counterclockwise along $\partial\mathbb{D}_j$. The homotopy class of the loop $\gamma_j:=\alpha^{-1}.\eta_{R,\mathbb{D}_j}.\alpha$ is called a \textit{meridian} of $a_j$ in $\CC\smallsetminus\Delta$. Then, the collections of meridians in $\CC\smallsetminus\Delta$ (one for each point of $\Delta$) define bases of $\pi_1(\CC\smallsetminus\Delta,\star).$ Similarly we can choose meridians $g_1,g_2$ in a fiber of the restriction (\ref{fibracion_localmente_trivial}) and $\gamma$ a meridian around the straight line $x=0$. It follows that
$$
\pi_1( (\CC^*\times \CC) \smallsetminus \mathcal{C})=\Big\langle g_1,g_2,\gamma;\ g_i^{\sigma^r}=g_i\;\; \wedge \;\; g_i^{\sigma^b}=\gamma^{-1}g_i\gamma\Big\rangle,
$$
where  $g_i^{\sigma^r}$ is the action on $g_i$ of $\gamma_j$ (known as the factorization of the braid monodromy, see \cite{ACT}).

The expression of the group can be simplified obtaining:

$$\pi_1((\CC^*\times \CC ) \smallsetminus \mathcal{C})=
\left\{
  \begin{array}{ll}
    \Big\langle \alpha, \beta,\gamma: \beta\alpha^r=\alpha^r\beta\; \wedge\; \gamma\alpha=\alpha\gamma \Big\rangle, & \hbox{if $r=2m+1$ is odd;} \\\\
    \Big\langle \alpha, \beta,\gamma: \alpha^r=\beta^2\; \wedge\; \gamma\alpha=\alpha\gamma \Big\rangle, & \hbox{if $r=2m$ is even.}
  \end{array}
\right.
$$
where $\alpha:=g_2g_1$ and $\beta:=(g_2g_1)^mg_2$.

\section{Holonomy of the essential component}\label{section_holonomy_of_the_essential_component}
The dynamical behavior of one foliation may be studied in a neighbourhood of a leaf by the representation of its fundamental group, as introduced by C. Ehresmann in 1950. In this study we will mainly follow J.-F. Mattei and R. Moussu \cite{Mattei-Moussu}: let us choose a point ${\bf p}$ of a leaf $L$ and a germ of a transversal section $\Sigma$ in ${\bf p}$. The lifting of a closed path $\gamma$ starting in ${\bf p}$, following the leaves of the foliation, induces germs of diffeomorphisms
$$h_{\gamma}:(\Sigma,{\bf p})\to (\Sigma,{\bf p}),$$
that only depend on the homotopy class of the path. The map $h_{\gamma}$ is called {\it holonomy } of the leaf $L$. The representation of the holonomy of $\pi_1(L,{\bf p})$ is the morphism defined by

$$\begin{array}{cccc}
Hol(L)&:\pi_1(L,{\bf p})&\to         & Diff(\Sigma,{\bf p})\\
      & \gamma          &\longmapsto & h_{\gamma},
\end{array}$$
and the {\it holonomy group} of the foliation along $L$ is the image of this application $Hol(L)$. Different points in the leaf and different transversal sections define conjugated  representations.

In order to lift the path $\gamma$ it is necessary to have a fibration, that is transverse to the foliation. Let us describe it. In general let $\Omega$ be an integrable 1-form in $(\CC^3,{\bf 0})$, $\pi:(M,E)\to (\CC^3,{\bf 0})$ a minimal reduction of singularities of the foliation $\FF_{\Omega}$. Let $\widetilde{\FF}_{\Omega}$ be the strict transform of $\FF_{\Omega}$ under $\pi$ and $D$ a component of the exceptional divisor $E$.

\begin{defin}\label{Def_Hopf_fibration}
A {\bf Hopf fibration adapted to $\FF_{\Omega}$}, $\mathcal{H}_{\FF_{\Omega}}$, is a holomorphic fibration $f:M\to D$ transverse to the foliation $\FF_{\Omega}$, i.e.,
\begin{enumerate}
  \item $f$ is a submersion and $f|_{D}=Id_{D}.$
  \item The fibres $f^{-1}(p)$ of $\mathcal{H}_{\FF_{\Omega}}$ are contained in the separatrices of $\widetilde{\FF}_{\Omega}$, for all $p\in D\cap\Sing(\widetilde{\FF}_{\Omega})$.
  \item The fibres $f^{-1}(p)$ of $\mathcal{H}_{\FF_{\Omega}}$ are transverse to the foliation $\FF_{\Omega}$, for all $p\in D\smallsetminus\Sing(\widetilde{\FF}_{\Omega})$.
\end{enumerate}
\end{defin}
 Consider now the particular case $\Omega\in\Sigma_{p,q}^{(d_1,\cdots,d_l)}$. The vector field
$$\mathcal{X}=px{\partial\over\partial x}+qy{\partial\over \partial y} + {pqd\over 2}z{\partial\over\partial z}. $$
verifies $\Omega(X)=pqd(z^2+\varphi)$, so, it is transverse to $\FF_{\Omega}$ outside the separatrix $S$, and $S$ is a union of trajectories of $\mathcal{X}$ (equivalently, $S$ is invariant by $\mathcal{X}$). The trajectories $\mathcal{X}$ give the fibres of the Hopf fibration adapted to $\Omega$.

Let $\widetilde{D}$ be the essential component of the reduction of singularities of $\FF_{\Omega}$, as defined in Section \ref{Desingularization}.

\begin{defin}
The \textbf{exceptional holonomy group} is the group of holonomy of $\widetilde{D}$. Let us denote it $H_{\Omega,\widetilde{D}}$.
\end{defin}

As a consequence of the results of Section \ref{section_toplogy_of_the_divisor}, $H_{\Omega,\widetilde{D}}$, is generated by elements $h_{\alpha}, h_{\beta}, h_{\gamma}$, which are the holonomy diffeomorphisms of, respectively, paths $\alpha, \beta, \gamma$.

Consider, now, two such foliations $\FF_{\Omega_1}, \FF_{\Omega_2}$, where $\Omega_1,\Omega_2\in\Sigma_{p,q}^{(d_1,\cdots,d_l)}$, with respective exceptional holonomy groups
$$H_{\Omega_1,\widetilde{D}}=\langle h_{\alpha}^{1},h_{\beta}^{1},h_{\gamma}^{1}\rangle, H_{\Omega_2,\widetilde{D}}=\langle h_{\alpha}^{2},h_{\beta}^{2},h_{\gamma}^{2}\rangle$$
represented on a the transverse section. Suppose they are analytically conjugated by an element $\psi\in Diff(\CC,0)$ such that
$$\psi^*(h_{\alpha_i}^1):=\psi^{-1}h_{\alpha_i}^1\psi=h^2_{\alpha_i};\; \alpha_i\in\{\alpha,\beta, \gamma\}.$$
Then we have the following result.

\begin{lema}\label{lema-extension-de-holonomia--}
There exist a fibered diffeomorphism $\phi, \phi(x,{\bf p})=(\varphi(x,{\bf p}),{\bf p})$ between two open neighbourhoods $V_j$ of $\widetilde{D}$ in the space  $(M,E)$ such that
\begin{enumerate}
  \item $\phi$ sends the leaves of the foliation ${\widetilde{\FF}_{\Omega_1}}{|_{V_1}}$ to the leaves of the foliation ${\widetilde{\FF}_{\Omega_2}}{|_{V_2}}$.
  \item The restriction of $\phi$ to the transverse section $\Sigma$ is $\psi$.
\end{enumerate}
\end{lema}
\begin{proof}
Let ${\bf p}\in \mathcal{L}=\widetilde{D}\smallsetminus \mathcal{S}$ and $\gamma_{\bf p}\subset \mathcal{L}$ be a path from ${\bf p}$ to ${\bf p}_j$. We have that the foliation $\FF_{\Omega_j}$ is generically transverse to the Hopf fibration relative to $\widetilde{D}$, $\mathcal{H}_{\FF_{\Omega_j}}$, outside the strict transform of the separatrix $\widetilde{S}$ and of the components $D_{\alpha j}, \mathcal{A}_j$, such that $\widetilde{D}\cap D_{\alpha j}\neq\emptyset$, $\widetilde{D}\cap \mathcal{A}_j\neq\emptyset$.

The projection $(\CC,0)\times\mathcal{L}\to \mathcal{L}$, is locally a covering map. Then, for each point $(x,{\bf p})\in\CC\times\mathcal{L}$, with $|x|$ small enough, we can consider the lifting $\widetilde{\gamma}^1_{\bf p}$, following a leaf of the foliation $\widetilde{\FF}_{\Omega_1}$, of the path $\gamma_{\bf p}$, such that $\widetilde{\gamma}_{\bf p}^1(0)=(x,{\bf p}).$\\
If $(x_1,{\bf p}_1)=\widetilde{\gamma}^1_{\bf p}(1)$  is the end point of $\widetilde{\gamma}^1_{\bf p}$, we lift the path $\gamma^{-1}_{\bf p}$ to $\widetilde{\gamma}^2_{\bf p}$ in the foliation $\widetilde{\FF}_{\Omega_2}$ such that $\widetilde{\gamma}^{2}_{\bf p}(0)=(\psi(x_1),{\bf p}_2)$.

Let $(x_2,{\bf p})=\widetilde{\gamma}^{2}_{\bf p}(1)$. Define
$$\phi(x,{\bf p})=(x_2,{\bf p})$$
as $\phi_{|\Sigma_1}=\psi,\; \phi(x,{\bf p})$ does not depend on the chosen path.

Let us observe that we can define $\phi$ as close as we want to the points ${\bf q}\in\mathcal{S}$, in $\widetilde{D}$. In fact, $\widetilde{D}\smallsetminus \mathcal{S}\cong \CC^*\times\CC\smallsetminus \mathcal{C}$, so ${\bf q}\in \mathcal{C}$ or ${\bf q}\in\CC^*$. We can consider
a meridian relative to ${\bf q}$ with base point ${\bf p}$, (see Section \ref{section_toplogy_of_the_divisor}), and  define the path $\beta_{\bf p}=\alpha_{\bf q}\gamma_{\bf p}$ with starting point ${\bf p}$ and end point ${\bf p}_j$. The path $\gamma^{-1}_{\bf p}\beta_{\bf p}$ is a meridian relative to ${\bf q}$ with base point ${\bf p}_j$, so its homotopy class define an element of $\pi_1(\mathcal{L},{\bf p}_j)$
and
$$h^1_{\gamma_{\bf p}^{-1}\beta_{\bf p}}=\psi^*(\gamma_{\bf p}^{-1}\beta_{\bf p}).$$

Therefore $\phi$ extends to a neighbourhood of $\widetilde{D}\smallsetminus \mathcal{S}$.

On the other hand, the singular points ${\bf q}\in\mathcal{S}$ are points of intersection of  $\widetilde{D}$ with  the other components of the divisor, and with the separatrix. These points are of dimensional type two or three and in a neighbourhood of  all them, $\widetilde{D}$ is a strong separatrix\footnote{In dimension two, \textit{strong} separatrices around simple singular points are the ones corresponding to non-zero eigenvalues of the linear part. In higher dimension, the separatrix is called \textit{strong} if it corresponds to a strong separatrix in dimension two for a generic plane transversal section.} 
From \cite{Cerveau-Mozo}, it follows that the conjugation of the holonomy of $\widetilde{D}$ implies the conjugation of the reduced foliations in a neighbourhood of these points. This conjugation coincides, outside of the separatrix, with the diffeomorphism $\phi$, and in  consequence, $\phi$ is extended to a diffeomorphism around  $\widetilde{D}$.
\end{proof}

\section{Analytic classification}\label{section_cladification}

In this section we will study the analytic classification of quasi-homogeneous cuspidal foliations $\FF_{\Omega}$ in $(\CC^3,{\bf 0})$, using as the main tool the lifting of the projective holonomy, as stated in previous sections.

As described in Section \ref{section_toplogy_of_the_divisor}, after the reduction of singularities of a quasi-homogeneous cuspidal holomorphic foliation in $(\CC^3,{\bf 0})$, the first component, $D_1$, of the exceptional divisor $E$, is either
\begin{itemize}
  \item $D_1\smallsetminus \mathcal{S}\simeq \CC^*\times\CC^*$, if the separatrix of $\FF_{\Omega}$ is defined by the equation
$$z^2+x^{n_1}y^{n_2}\prod_{i=1}^{l}\big(y^{p}-a_ix^{q}\big)^{d_i}=0,\ \text{with } (n_1,n_2)\neq (0,0),$$
  \item or $D_1\smallsetminus \mathcal{S}\simeq \CC^2$,  if the separatrix of $\FF_{\Omega}$ is defined by the equation
$$z^2+\prod_{i=1}^{l}\big(y^{p}-a_ix^{q}\big)^{d_i}=0.$$
\end{itemize}

In the admissible case that we are studying, $D_1\smallsetminus \mathcal{S}$ is simply connected, so, the existence of a first integral around it is guaranteed due to the results of Mattei and Moussu \cite{Mattei-Moussu}. In order to be able to extend this first integral, and consequently, to extend the conjugation diffeomorphism, it would be necessary to impose additional technical conditions on one, or possibly several components of the exceptional divisor, that we will be call in the sequel \textit{special components}. These special components will be
those that arise at the end of the monoidal transformations with center the projective lines $D_1\cap S_1$, $D_{c_0+1}\cap S_{c_0+1}$ that we will denote by $\widetilde{D}', \;\widetilde{D}''$ respectively. Note that
$$\widetilde{D}'\smallsetminus \mathcal{S}\simeq\CC\times(\CC\smallsetminus\{2pts\})\approx \widetilde{D}''\smallsetminus \mathcal{S}.$$
This will motivate the following definition:

\begin{defin}\label{defini_propiedades_p_1andp_2}
Let $\Omega\in\Sigma_{p,q}^{{(d_1,\cdots,d_l)}}$ be, we will say that the foliation $\FF_{\Omega}$:

\begin{enumerate}
  \item For $d-$even ({\bf Case i}),  satisfies the property $\wp_1$: if the holonomy of the leaves $\widetilde{D}'\smallsetminus \mathcal{S},\; \widetilde{D}''\smallsetminus \mathcal{S}$, is linearizable.
  \item For $d-$odd ({\bf Case ii.1}),  satisfies the property $\wp_2$: if the holonomy of the leaf $\widetilde{D}''\smallsetminus \mathcal{S}$, is linearizable.
\end{enumerate}
\end{defin}

The following theorem is the main result of this paper.

\begin{teorema}\label{thm_clasificacion_analitica}
Let $\Omega_1, \Omega_2$ be elements of $\Sigma_{p,q}^{(d_1,\cdots,d_l)}$. Consider the foliations $\FF_{\Omega_1}$ and $\FF_{\Omega_2}$ that satisfy one of the properties $\wp_1, \wp_2$ if we are in one of the cases described above, with exceptional holonomy groups $H_{\Omega_i, \tilde{D}}=\langle h_{g_1}^i,h_{g_2}^i,h_{\alpha}^i\rangle, \; i=1,2$. Then, the foliations are analytically conjugated if and only if the triples $(h_{g_1}^i,h_{g_2}^i,h_{\alpha}^i)$ are also analytically conjugated.
\end{teorema}

\begin{proof}
If the foliations are conjugated, then their essential holonomy groups are conjugated. The arguments are exactly the same that those described in  \cite{Cerveau-Mozo}.

Assume that the holonomies are conjugated via $\Psi$, and let $\widetilde{\FF}_{{\Omega}_i}$ be the respective strict transform of the foliations $\FF_{\Omega_i}$, after its reduction of singularities. From the existence of the Hopf fibration relative to $\widetilde{D}$ (see Figure \ref{secuencia-explosiones-segun-di}) and Lemma \ref{lema-extension-de-holonomia--}, $\Psi$ may be extended to a neighbourhood $V_i\subset (M,E)$ of $\widetilde{D}$ and as a  consequence, we have that $\widetilde{\FF}_{{\Omega}_i}$ are conjugated in $V_i$. Now, we need to conjugate the foliations in a neighbourhood of all the other components of the divisor. We must first check the existence of the first integral around of the singular points outside $\widetilde{D}$. In fact, note that $D_1\smallsetminus \mathcal{S}$ is simply connected, so its holonomy group is trivial. As a consequence, the holonomy of the leaf $D_{2}\smallsetminus \mathcal{S}\approx \CC\times\CC^{*}$, generated by a loop  around $L_1:=D_1\cap D_{2}$, is periodic (see \cite{Mattei-Moussu}). The same argument proves that $D_{\alpha}\smallsetminus \mathcal{S}$ has periodic holonomy, for all $1<\alpha\leq c_0+1$. So, $\widetilde{\FF}_{{\Omega}_i}$  has first integral around
$$L_{\alpha}:=D_{\alpha}\cap D_{\alpha+1},\;1\leq\alpha \leq c_0+1.$$

On the other hand, the leaves $D_{\alpha}\smallsetminus \mathcal{S}\thickapprox \CC^*\times\CC^*$ have periodic holonomy, generated by two loops around the lines
$$L_{\alpha}=D_{\alpha}\cap D_{\alpha+1};$$
$$L_{s_\nu}^{\alpha}=D_{\alpha}\cap D_{s_\nu}.$$

The above arguments prove the existence of the first integral around the lines
$$L_{\alpha j}:=D_{\alpha j}\cap D_{\alpha(j-1)}$$
where $\alpha, j$ are such that $D_{\alpha j}\smallsetminus \mathcal{S}\thickapprox\CC\times\CC^*$ or $\thickapprox \CC^*\times\CC^*$.

Finally, we need to guarantee the existence of the first integral around the points that represent the intersection of the divisor with the strict transform of the separatrix. These points are in the components $D_{\alpha j}$, and:

\begin{itemize}
  \item Either $D_{\alpha j}\smallsetminus \mathcal{S}\thickapprox \CC\times(\CC\smallsetminus\{{\rm 2\;points}\})$,
  \item or $D_{\alpha j}\smallsetminus \mathcal{S}\thickapprox \CC^*\times(\CC\smallsetminus\{{\rm 2\;points}\})$,
  \item or $D_{\alpha j}\smallsetminus \mathcal{S}\thickapprox (\CC^*\times\CC )\smallsetminus \mathcal{C}$.
\end{itemize}

Let us denote by $\mathcal{D}:=D_{\alpha j}$ such that $D_{\alpha j}\smallsetminus \mathcal{S}\thickapprox \CC\times(\CC\smallsetminus\{{\rm 2\;points}\})$.
Note that
$\mathcal{D}\cap\widetilde{S}=\mathcal{L}_1\cup \mathcal{L}_2$ or $\mathcal{D}\cap\widetilde{S}=\mathcal{L}$, where $\widetilde{S}$ represents the strict transform of the separatrix, and $\mathcal{L},\mathcal{L}_1,\mathcal{L}_2$ are projective lines.

Let us suppose now that $d$ is even, then there exist exactly two components of the exceptional divisor such that $\mathcal{D}\cap\widetilde{S}=\mathcal{L}_1\cup \mathcal{L}_2$. The points of $\mathcal{L}_1, \mathcal{L}_2$ are singular points of dimensional type two.

Assuming that $\Omega_j\in\Sigma_{p,q}^{(d_1,\cdots,d_l)}$, $\FF_{\Omega_j}$ satisfies the property $\wp_1$, i.e., the holonomy of the leaves $\mathcal{D}\smallsetminus \mathcal{S}$ is linearizable. This fact guarantees the existence of the first integral around $\mathcal{L}_1,\; \mathcal{L}_2$.

\begin{center}
 \includegraphics[width=10cm]{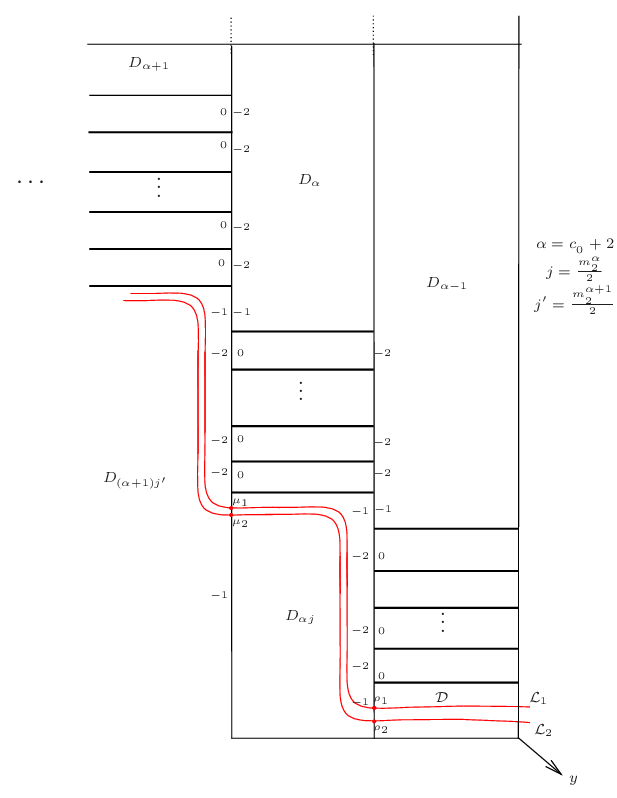}
 \figcaption{Existence of the first integral around the points $\rho_i$ and $\mu_i$}\label{sin nombre}
\end{center}

Now let us prove the existence of the first integral around the points (see Figure \ref{sin nombre}) $$\rho_i:=D_{\alpha j}\cap\widetilde{S}\cap\mathcal{D}.$$
Let $h_{\alpha_i}$ be the holonomy associated to $\alpha_i$, loop around $\mathcal{L}_i$; $h_{\alpha_i}$ is linearizable. Now let us consider $\beta\subset\mathcal{D}\smallsetminus\mathcal{S}\simeq \CC\times(\CC\smallsetminus\{2\;{\rm points}\})$, a loop around $D_{\alpha j}\cap\mathcal{D}$. Note that $\beta\subset \CC\times\{1\;{\rm point}\}\simeq\CC$; so, $\beta$ is homotopically trivial and the  associated holonomy to $\beta$ is $1_{\CC}$, the identity map, $h_{\beta}=1_{\CC}$. Then, the holonomy group of $D_{\alpha j}\smallsetminus \mathcal{S}$ is linearizable and therefore, around $\rho_i$,  $\widetilde{\Omega}_i$ is linearizable. So, there exists a first integral around $\rho_i$. Consider now the points $\mu_i$ (see Figure \ref{sin nombre}), and $h_\gamma$, the holonomy associated to $\gamma$, loop around the projective line $D_{(\alpha+1)j}\cap D_{\alpha j}$. Note that $\gamma$ is such that $\gamma^{-1}$ is a loop around $D_{\alpha-1}\cap D_{\alpha j}$, so $h_{\gamma^{-1}}$ is the holonomy of the leaf $D_{\alpha-1}\smallsetminus \mathcal{S}$, which is periodic. Then around $\mu_i$ there exists a first integral. In a similar way, a first integral can be found around $E\cap \widetilde{S}$.

Let us suppose that $d,q$ are odd and $p$ is even. There exist exactly a component $\mathcal{D}$ for which $\mathcal{D}\cap\widetilde{S}=\mathcal{L}_1\cup \mathcal{L}_2$ and a finite number of components for which we have $\mathcal{D}\cap\widetilde{S}=\mathcal{L}$. It is easy to see that around $\mathcal{L}$ there exist a first integral, and the existence of first integral around the lines $\mathcal{L}_1, \mathcal{L}_2$ follows, from property $\wp_2$, as in the previous case.

Let us see now the extension of $\Psi$ around all the exceptional divisor. The idea is, first, to extend $\Psi$ to a neighborhood of all the components of the divisor in which the separatrix intersects, and finally to extend it to the rest of the components. In both situations we should respect the fibration and the first integral that we have constructed previously. We have that $\widetilde{\FF}_{{\Omega}_j}$ are  analytically conjugated, via $\Psi$, in a neighbourhood of $\widetilde{D}$. We want to extend $\Psi$ to a neighbourhood of the exceptional divisor.
$$E=E_I\cup E_{II}\cup E_{III}.$$

Note that $E_{III}=\bigcup\mathcal{A}_j\times\CC$, and these components are topologically equivalent to ${\mathbb P}_{\CC}^1\times\mathbb{D}$. As $\widetilde{\FF}_{{\Omega}_j}$ are analytically conjugate in a neighbourhood of $\widetilde{D}$, it follows that $\Psi$ is extended to a neighbourhood of $\widetilde{D}\cup E_{III}$

\begin{figure}[h]
\begin{center}
  \includegraphics[width=11.5cm]{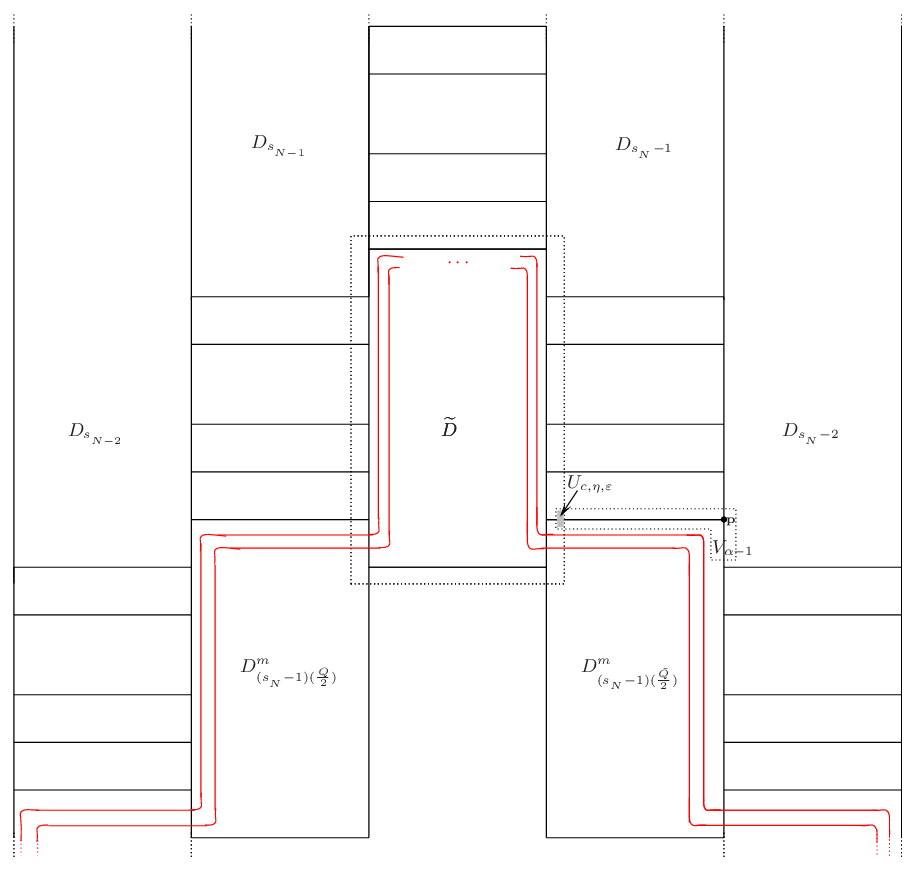}\\
 \figcaption{Extension of $\Psi=\Psi_2\Theta\Psi_1^{-1}$ to a neighbourhood of $D_{(\alpha-1)\beta}\cap D_{(\alpha-1)(\beta-1)}$.}\label{extension-analitica}
\end{center}
\end{figure}

On the other hand, for the case that $d$ is even, around ${\bf p}:=D_{\alpha-2}\cap D_{(\alpha-1)\beta}\cap D_{(\alpha-1)(\beta-1)}$, with $\alpha=s_{_N}$ and $\beta={\tilde{Q}_2\over2}$ (where $\tilde{Q}= \left( \frac{pq}{\delta} -nq\right) d-2\left( \frac{p+q}{\delta} -(m+n+1)\right)$, see Figure \ref{extension-analitica}) a first integral is given by the equation
$$F_j:= x^{m_{\nu}}s^{n_{\nu}+2(\epsilon-1)}t^{n_{\nu}-2\epsilon}U_j(x,s,t);\  U_j({\bf 0})\neq0$$
where for simplicity we denote $x:=x_{\alpha-1};\; \epsilon=\beta-1,\; \nu=N\;; j_{\nu}=c_{\nu}-1$.
Let $$\mathcal{B}_c=\Big\{x\in\CC: |x|<c\Big\}$$ be, for $c\in\mathbb{R}^{+}$ large enough, and
$$D_{\varepsilon}:=\{s\in\CC:|s|<\varepsilon_1\}\times\{t\in \CC: |t|<\varepsilon_2\}.$$
Similar arguments as in \cite{Meziani} and \cite{FM}  are now used. First define a diffeomorphism $\Psi_j=(\Psi_{j1},\Psi_{j2},\Psi_{j3})$, on an open set $\mathcal{B}_c\times D_{\varepsilon}$, that transforms the first integral in $x^{m_{\nu}}s^{n_{\nu}+2(\epsilon-1)}t^{n_{\nu}-2\epsilon}$ and respects the fibration, i.e.:
\begin{equation}\label{trasforma-Int.Primera}
\Psi_{j1}^{m_{\nu}}\Psi_{j2}^{n_{\nu}+2(\epsilon-1)}\Psi_{j3}^{n_{\nu}-2\epsilon}=x^{m_{\nu}}s^{n_{\nu}+2(\epsilon-1)}t^{n_{\nu}-2\epsilon}U_j,
\end{equation}
and
\begin{equation}\label{respeta-fibracion}
\begin{array}{ll}
    \Psi_{j3}\Psi_{j1}^{{P\over2}-\widetilde{Q}_2+2}& =tx^{{P\over2}-\widetilde{Q}_2+2}\\\\
    \Psi_{j2}\Psi_{j1}^{\widetilde{Q}_2-{P\over2}} & =sx^{\widetilde{Q}_2-{P\over2}}.
\end{array}
\end{equation}
Consider now $\Theta:=\Psi_2\circ\Psi\circ\Psi_1^{-1}$, defined in the open set $U_{c,\eta,\varepsilon}$. $\Theta$ sends $\Psi_1(\widetilde{\FF}_{{\Omega}_1})$  on $\Psi_2(\widetilde{\FF}_{{\Omega}_2})$, and respects the first integral $$x^{m_{\nu}}s^{n_\nu+2(\epsilon-1)}t^{n_\nu-2\epsilon}.$$
It is defined on a set of the type
$${|x^{m_{\nu}}s^{n_\nu+2(\epsilon-1)}t^{n_\nu-2\epsilon}|<\varepsilon},$$ in the considered charts, and this set intersects the domain of definition of $\Psi$. So $\Psi=\Psi_2\circ\Theta\circ\Psi_1^{-1}$ can be extended to a neighbourhood of $D_{(\alpha-1)\beta}\cap D_{(\alpha-1)(\beta-1)}$.
Repeat the same argument to extend $\Psi$ to a neighbourhood of
$$\big(D_{(\alpha-1)\beta}\cap D_{\alpha_2}\big)\cup\Big(D_{(\alpha-1)\beta}\cap(\cup D_{(\alpha-2)j})\Big)\cup (D_{(\alpha-1)\beta}\cap \widetilde{S}),$$
and in consequence, $\Psi$ can be extended to a neighbourhood of $D_{(\alpha-1)\beta}$. This process can be repeated and extend $\Psi$ to a neighbourhood of the exceptional divisor $E$. So $\FF_{\Omega_i}$ are analytically conjugated, outside of the singular locus of codimension two. Finally, using Hartogs theorem we can obtain the extension of the conjugation around of the origin.

In the case $d, q$ odd and $p$ even, we proceed analogously.

 \end{proof}

 \section*{Acknowledgments}
The authors want to thank the Pontificia Universidad Cat\'{o}lica del Per\'{u} and the Universidad de Valladolid for their hospitality during the visits while preparing this paper.

The authors would like to thank the anonymous referee for many valuable and constructive suggestions, that have helped to improve the paper.

\end{document}